\documentclass[12pt, a4paper]{amsart}

\usepackage{amsmath,amssymb,amsfonts}
\usepackage{bbold}
\usepackage{epic,eepic}
\usepackage{enumerate}

\usepackage{amsthm}
\usepackage[all]{xy}
\usepackage{caption}
\usepackage{stmaryrd} 
\usepackage[dvipdfmx]{graphicx,color} 
\usepackage{tikz}
\usetikzlibrary{trees}

\theoremstyle{plain}
\newtheorem{thm}{Theorem}[section]
\newtheorem{prop}[thm]{Proposition}
\newtheorem{lem}[thm]{Lemma}

\theoremstyle{definition}
\newtheorem{defn}[thm]{Definition}
\newtheorem{exmp}[thm]{Example}

\numberwithin{figure}{section}

\numberwithin{table}{section}

\DeclareMathOperator*{\esssup}{ess\,sup}

\newcommand{\lspace} {
  \vspace{0.8\baselineskip}
}

\newcommand{\pair}[2]{
  \langle  #1, \, #2 \rangle 
}

\newcommand{\abs}[1]{
  \lvert  #1 \rvert
}

\newdir{ >}{{}*!/-5pt/@{>}}

\definecolor{arrowred}{rgb}{0,0,0} 
\newcommand{\newword}[1]{\textbf{\textit{#1}}}

\numberwithin{equation}{section}

\newcommand{\ProbX}{\bar{X}}
\newcommand{\ProbY}{\bar{Y}}
\newcommand{\ProbZ}{\bar{Z}}
\newcommand{\ProbNX}{\bar{X}}
\newcommand{\ProbNY}{\bar{Y}}
\newcommand{\ProbNZ}{\bar{Z}}
\newcommand{\Prob}{\mathbf{Prob}}
\newcommand{\mpProb}{\mathbf{mpProb}}
\newcommand{\CPS}{\mathbf{CPS}}

\def \obsolete {0}

\newcommand{\Set}{\mathbf{Set}}

\mathchardef\mhyphen="2D  

\title {A Category of Probability Spaces}

\thanks{This work was supported by JSPS KAKENHI Grant Number 18K01551.}

\author[T. Adachi and Y. Ryu]{Takanori Adachi and Yoshihiro Ryu}
\address{Graduate School of Management,
         Tokyo Metropolitan University,
         1-4-1 Marunouchi, Chiyoda-ku, Tokyo 100-0005, Japan}
\email{Takanori Adachi <tadachi@tmu.ac.jp>}

\address{Department of Mathematical Sciences,
         Ritsumeikan University,
         1-1-1 Nojihigashi, Kusatsu, Shiga, 525-8577 Japan}
\email{Yoshihiro Ryu <iti2san@gmail.com>}

\date{\today}



\begin{document}

\maketitle

\begin{abstract}
We introduce a category $\Prob$ of probability spaces whose objects are 
all probability spaces
and whose arrows correspond to measurable functions
 satisfying an absolutely continuous requirement.
We can consider 
a $\Prob$-arrow
as
an evolving direction of information.
%
We introduce a contravariant functor
$\mathcal{E}$
from 
$\Prob$
to
$\Set$,
the category of sets.
The functor
$\mathcal{E}$
provides conditional expectations along arrows in 
$\Prob$,
which are generalizations of the classical conditional expectations.
For a 
$\Prob$-arrow
$f^-$,
we introduce two concepts
$f^-$-measurability
and
$f^-$-independence
and investigate their interaction with 
conditional expectations along
$f^-$.
We also show that
the completion of probability spaces
is naturally formulated as an endofunctor of
$\Prob$.

\end{abstract}



\section{Introduction}
\label{sec:intro}

One of the most prominent examples
of applying category theory to probability theory
 is 
Lawvere and
Giry's approach
of formulating transition probabilities in a monadic example
(
\cite{lawvere_1962},
\cite{giry_1982}
).
However,
there are few
of making categories
consisting of all
probability spaces
due to a difficulty of
finding an appropriate condition of their arrows.
One of the 
trials
is a way to adopt measure-preserving functions as arrows.
With this setting, for example, 
Franz develops a stochastic independence theory in
\cite{franz_2003}.
Our approach
is one of this simple-minded trials.
%
Another recent trial of generalizing arrows is made
by
Motoyama and Tanaka
\cite{MT_2016}.
They
introduce
a notion of 
bounded arrows between probability spaces
and define the category
of all probability spaces and all bounded arrows between them,
called $\CPS$.

%
We have two main results in this paper.
One is an introduction of
the category
$\Prob$
of all probability spaces and
null-presearving maps between them.
The other one is that we show the existence  of
the conditional expectation functor 
from 
$\Prob$
to
$\Set$,
which is a natural generalization of
the classical notion of conditional expectation.

We introduce a
category
$\Prob$
of all probability spaces
in order to see
a possible generalization of
some classical tools in probability theory
including conditional expectations.
Actually, 
\cite{adachi_2014crm}
provides a simple category
for formulating conditional expectations,
but its objects and arrows are so limited 
that we cannot use it as a foundation of categorical probability theory.
%
We will also see 
that 
all arrows in the category
$\CPS$
defined in
\cite{MT_2016}
are arrows in
$\Prob$
as well if ignoring they have opposite directions.
Therefore,
$\CPS^{op}$
is a subcategory of
$\Prob$.


The original idea of the category 
$\Prob$
comes when
we sought a generalization of the notion of financial risk measures
that is one of the crucial tools for managing risks in the financial industry.
The risk measure is a means of evaluating a future risk 
that is represented as a random variable,
with current information by calculating its conditional expectation given the information.
The reason of using the conditional expectation is 
that we have less information than we will have in future.
The conditional expectation is a perfect tool
as long as the only difference between today and future is the information we can access,
that is,
the changing part of a probability space
$(\Omega, \mathcal{F}, \mathbb{P})$
from now to future
is just the 
$\sigma$-field
$\mathcal{F}$.
However, after experiencing recent financial crises,  
we are suspecting that 
the probability measure
$\mathbb{P}$
also varies through time, 
which created the disasters since
we treated it as invariant when we calculated the risk.
A trial of making the probability measure vary was 
the motivation of 
\cite{adachi_2014crm}.

In this paper,
beyond that,
we treat the situation when the underlying set
$\Omega$
of elementary events
also
varies,
that is, all the three components of the probability space are changing.
We represent this change of entire probability spaces by an arrow between them,
thinking within a category of probability spaces.
So a natural requirement for the arrow is that
we can extend the classical conditional expectation given the (current) 
$\sigma$-field
to a sort-of conditional expectation along the arrow.
The category
$\Prob$
was developed so that this requirement is satisfied.


The arrows of
$\Prob$
are
maps
corresponding to
measurable functions
satisfying an absolutely continuous requirement
that is weaker than the measure-preserving requirement
(Section \ref{sec:cat_prob}).
The requirement can be restated the inverse of the arrow preserves null sets.
The resulting 
$\Prob$-arrow
can be seen as
an evolving direction of information
together with its interpretation.
We will see that
the requirement allows us to
extend some important notions
relativized by a $\sigma$-algebra
in classical probability theory
to notions relativized by
a $\Prob$-arrow
$f^-$.
For example,
we introduce notions of
a conditional expectation along 
$f^-$
(Section \ref{sec:CE_fun}),
a
$f^-$-measurable function
(Section \ref{sec:meas})
and
a random variable independent of
$f^-$
(Section \ref{sec:indep}).
These are considered 
as generalizations of 
the classical counterparts.
The existence of those natural generalizations may
support a claim saying that
the requirement for 
$\Prob$-arrows
is a natural one.
We also see that
the completion procedure of probability spaces
becomes an endofunctor
of
$\Prob$
(Section \ref{sec:complF}).

The category
$\Prob$
and functors developed in this paper
convey more natural and richer structures
than those introduced in
\cite{adachi_2014crm}.

\section{Category of Probability Spaces}
\label{sec:cat_prob}

In this paper,
$
\bar{X} =
(X, \Sigma_X, \mathbb{P}_X)
$,
$
\bar{Y} =
(Y, \Sigma_Y, \mathbb{P}_Y)
$
and
$
\bar{Z} =
(Z, \Sigma_Z, \mathbb{P}_Z)
$
are
probability spaces.

\begin{defn}{[Null-preserving functions]}
A measurable function
$
f
  :
\ProbY
  \to
\ProbX
$
is called
\newword{null-preserving}
if
$
f^{-1}(A)
\in \mathcal{N}_Y
$
for every
$
A \in \mathcal{N}_X
$,
where
$
\mathcal{N}_X
  :=
\mathbb{P}_X^{-1}(0)
  \subset
\Sigma_X
$
and
$
\mathcal{N}_Y
  :=
\mathbb{P}_Y^{-1}(0)
  \subset
\Sigma_Y
$.
\end{defn}

The following characterization is straightforward.

\begin{prop}
\label{prop:null_pres}
Let
$
f
  :
\ProbY
  \to
\ProbX
$
be a measurable function.
Then,
$f$
is null-preserving iff
$
\mathbb{P}_Y \circ f^{-1} \ll \mathbb{P}_X
$,
where
$
\mu
  \ll
\nu
$
means that
$\mu$
is 
\newword{absolutely continuous}
 with respect to
$\nu$,
that is,
$
\mu(A) = 0
$
whenever
$
\nu(A) = 0
$.
\end{prop}

The following diagram 
(that does not commute in general)
might be helpful to see the situation we consider.
\begin{equation*}
\xymatrix@C=20 pt@R=10 pt{
  Y
     \ar @{->}_{f} [dd]
&
  \mathcal{N}_Y
     \ar @{}^-{\subset} @<-6pt> [r]
&
  \Sigma_Y
     \ar @{->}^{\mathbb{P}_Y} [dr]
\\
&&&
  [0, 1]
\\
  X
&
  \mathcal{N}_X
     \ar @{->}^{f^{-1}} [uu]
     \ar @{}^-{\subset} @<-6pt> [r]
&
  \Sigma_X
     \ar @{->}^{f^{-1}} [uu]
     \ar @{->}_{\mathbb{P}_X} [ur]
}
\end{equation*}

\begin{defn}{[Bounded functions (Motoyama-Tanaka \cite{MT_2016})]}
A measurable function
$
f
  :
\ProbY
  \to
\ProbX
$
is called
\newword{bounded}
if
there exists a positive number
$
M > 0
$
such that 
$
\mathbb{P}_Y(f^{-1}(A))
	\le
M 
\mathbb{P}_X(A)
$
for every
$
A
	\in
\Sigma_X
$.
\end{defn}

By Proposition \ref{prop:null_pres},
the following proposition is obvious.

\begin{prop}
\label{prop:bnded_fct}
Every bounded function
$
f
  :
\ProbY
  \to
\ProbX
$
is null-preserving.
\end{prop}


\begin{prop}
\label{prop:comp_prob_arrows}
Let
$f$
and
$g$
be two null-preserving functions as follows:
\begin{equation*}
\xymatrix{
  \ProbZ
       \ar @{->}^{g} [r]
 &
  \ProbY
       \ar @{->}^{f} [r]
 &
  \ProbX
}
\end{equation*}
Then,
$
f \circ g
$
is also null-preserving.
\end{prop}
\begin{proof}
Immediate.
\end{proof}

Proposition \ref{prop:comp_prob_arrows}
makes the following definition
well-defined.

\begin{defn}{[Category $\Prob$]}
\label{defn:cat_prob}
A category
$
\Prob
$
is the category whose objects are
all probability spaces and 
the set of arrows between them are defined by
\begin{align*}
\Prob(
  \ProbX,
&
  \ProbY
)
  :=
\{
f^-
  \mid
f :
  \ProbY
\to
  \ProbX
\textrm{ is a null-preserving function.}
\},
\end{align*}
where
$f^-$
is a symbol
corresponding uniquely to
a function $f$.

We write 
$Id_X$
for an identity \textit{measurable} function 
from
$\ProbX$
to
$\ProbX$,
while writing
$id_X$
for an identity function
from
$X$
to
$X$.
Therefore,
the identity arrow of
a
$\Prob$-object
$\ProbX$
is 
$Id_X^-$.

\end{defn}

Motoyama and Tanaka \cite{MT_2016}
introduce the category
consisting of
all probability spaces and
all bounded arrows between them,
and call it
$\CPS$.

By 
Proposition \ref{prop:bnded_fct},
$\CPS^{op}$
is a subcategory
of
$\Prob$.
So the category
$\CPS$
is,
in a sense,
not large enough 
for
developing 
the theory of financial risk measures
as
we mentioned in 
Section \ref{sec:intro}.

\lspace

An arrow 
$f^-$
in
$\Prob$
can be considered to represent
an evolving direction of information
with a way of its interpretation.
The information is evolving along 
$f^{-1}$
but with a restriction
to its accompanying probability measure.

In order to see it more concretely,
let us consider the case where 
$f$ is an identity function on $X$, 
that is,
consider a 
$\Prob$-arrow
$
id_{X}^-
  :
(X, \Sigma_1, \mathbb{P}_1)
  \to
(X, \Sigma_2, \mathbb{P}_2)
$.
Then, 
we have
$
\Sigma_1
  \subset
\Sigma_2
$
and
$
\mathbb{P}_2
  =
\mathbb{P}_2
  \circ
id_X^{-1}
  \ll
\mathbb{P}_1
$.
This means that
the information is growing
while
the support of the probability measure is decreasing.
The latter makes sense
if we think of the following situation:
someone believed that some event
among many other events 
 may happen,
but now she has changed her mind 
to believe that the event will never occur,
and so she can concentrate on other possible events.

Actually,
\cite{adachi_2014crm}
treats this special situation.
It
introduces,
for a given measurable space
$
(\Omega, \mathcal{G})
$,
a category
$
\chi(\Omega, \mathcal{G})
$
whose objects
are all the pairs of the form
$
(\mathcal{F}_{\mathcal{U}},
\mathbb{P}_{\mathcal{U}})
$
where
$\mathcal{F}_{\mathcal{U}}$
 is 
a sub-$\sigma$-field
of
$\mathcal{G}$ and
$\mathbb{P}_{\mathcal{U}}$
is a probability measure on 
$\mathcal{G}$.
And it has
 a unique arrow
from 
$
(\mathcal{F}_{\mathcal{V}},
\mathbb{P}_{\mathcal{V}})
$
to
$
(\mathcal{F}_{\mathcal{U}},
\mathbb{P}_{\mathcal{U}})
$
only when
$
\mathcal{F}_{\mathcal{V}}
  \subset
\mathcal{F}_{\mathcal{U}}
$
and
$
\mathbb{P}_{\mathcal{U}}
  \ll
\mathbb{P}_{\mathcal{V}}
$.
Note that
there is a natural embedding
$\iota$
 of the category 
$
\chi(\Omega, \mathcal{G})
$
into
$\Prob$.
\begin{equation*}
\xymatrix@C=40 pt@R=10 pt{
  \chi(\Omega, \mathcal{G})
     \ar @{->}^{\iota} [r]
&
  \Prob
\\
  (\mathcal{F}_{\mathcal{V}}, \mathbb{P}_{\mathcal{V}})
     \ar @{->}^{\iota} [r]
     \ar @{->}_{*} [dd]
&
  (\Omega, \mathcal{F}_{\mathcal{V}}, \mathbb{P}_{\mathcal{V}}
         \mid
    \mathcal{F}_{\mathcal{V}}
  )
     \ar @{->}^{
       \iota(*) := id_{\Omega}^-
     } [dd]
\\
\\
  (\mathcal{F}_{\mathcal{U}}, \mathbb{P}_{\mathcal{U}})
     \ar @{->}^{\iota} [r]
&
  (\Omega, \mathcal{F}_{\mathcal{U}}, \mathbb{P}_{\mathcal{U}}
         \mid
    \mathcal{F}_{\mathcal{U}}
  )
}
\end{equation*}

\begin{prop}
\label{prop:init_obj}
A probability space
$
\mathbb{0}
  :=
(
  \{*\},
  \{ \{*\}, \emptyset \},
  \mathbb{P}_{\mathbb{0}}
)
$,
where
$
  \mathbb{P}_{\mathbb{0}}(\{*\}) := 1
$
and
$
  \mathbb{P}_{\mathbb{0}}(\emptyset) := 0
$,
is an initial object of the category
$\Prob$.
Actually, for a probability space
$
\ProbX
$,
$
!_{X}^-
  :
\mathbb{0}
  \to
\ProbX
$
is a unique arrow in
$\Prob$,
where
$
!_X
  :
X
  \to
\{ * \}
$
is a function such as
$
!_X(x) = *
$
for all 
$
x \in X
$.
\end{prop}
\begin{proof}
First, we show the uniqueness of
$!_X^-$.
But it is a straightforward consequence from the fact that
there exists only one arrow
$!_X$
from 
$X$
to
$\{*\}$.
Next, we prove that
$!_X^-$
is a
$\Prob$-arrow.
\if \obsolete 1
{\color{red}
All we need to show is that
$
\mathbb{P}_X \circ !_X^{-1}
  \ll
\mathbb{P}_{\mathbb{0}}
$.} 
\fi 
Obviously $!_X$ is measurable, so all we need to show is that $!_X^{-1}$ is null-preserving.
\if \obsolete 1
{\color{red}
Assume that
$
\mathbb{P}_{\mathbb{0}}(A) = 0
$
for some
$
A \in 
\Sigma_{\mathbb{0}}
  :=
\{
   \{*\},
   \emptyset
\}
$.
Then
$A$
must be
$\emptyset$.
Therefore
$
\mathbb{P}_X \circ !_X^{-1}(A)
  =
\mathbb{P}_X( !_X^{-1}(\emptyset))
  =
0
$.} 
\fi 
Since $\emptyset$ is the only null set of $\mathbb{0}$ and any inverse image of $\emptyset$ is also $\emptyset$, we conclude $!_X^{-1}$ is null-preserving.
\end{proof}


In the following discussions,
we fix the state space to be
the measurable space
$
(
 \mathbb{R}, 
 \mathcal{B}(\mathbb{R})
)
$
for simplicity.
$
\mathcal{L}^{\infty}(\ProbNX)
$
is a vector space consisting of 
$\mathbb{R}$-valued random variables
$v$
such that
$
\mathbb{P}_X\mhyphen\esssup_{x \in X}
\abs{v(x)}
  <
\infty
$,
while
$
\mathcal{L}^1(\ProbNX)
$
is a vector space consisting of 
$\mathbb{R}$-valued random variables
$v$
such that
$
\int_X
  \abs{v}
\, d \mathbb{P}_X
$
has a finite value.
For
two random variables
$
u_1
$
and
$
u_2
$,
we write
$
u_1
  \sim_{\mathbb{P}_X}
u_2
$
\; or \;
$
u_1 = u_2
\; \mathbb{P}_X\mhyphen\mathrm{a.s.}
$
when
$
\mathbb{P}_X(
  u_1 \ne u_2
)
  =
0
$.
$
    L^{\infty}(\ProbNX)
$
and
$
    L^1(\ProbNX)
$
are quotient spaces
$
  \mathcal{L}^{\infty}(\ProbNX)
/
  \sim_{\mathbb{P}_X}
$
and
$
  \mathcal{L}^1(\ProbNX)
/
  \sim_{\mathbb{P}_X}
$,
respectively.

\begin{prop}
\label{prop:ar_congr}
Let
$
u_1
$
and
$
u_2 
$
be two elements of
$
  \mathcal{L}^{\infty}(\ProbNX)
$,
and
$
f^-
$
be an arrow
  in
$
\Prob(
  \ProbX,
  \ProbY
)
$.
Then,
$
u_1
  \sim_{\mathbb{P}_X}
u_2
$
implies
$
u_1 \circ f
  \sim_{\mathbb{P}_Y}
u_2 \circ f
$.
\end{prop}
\begin{proof}
Assume that
$
u_1
  \sim_{\mathbb{P}_X}
u_2
$.
Then,
$
\mathbb{P}_X(
  u_1 \ne u_2
)
  = 0
$.
%
%
Hence, we have
$
\mathbb{P}_Y\big(
	f^{-1} \{u_1 \ne u_2\}
\big)
	=
(\mathbb{P}_Y \circ f^{-1})
  (u_1 \ne u_2)
	=
0
$
since
$
  \mathbb{P}_Y \circ f^{-1} \ll \mathbb{P}_X
$.
Therefore
$
\mathbb{P}_Y(
  u_1 \circ f
   \ne
  u_2 \circ f
)
  = 0
$
since
$
\{
   u_1 \circ f
\ne
   u_2 \circ f
\}
  \subset
f^{-1}
\{
  u_1 \ne u_2
\}
$,
which means
$
u_1 \circ f
  \sim_{\mathbb{P}_Y}
u_2 \circ f
$.
\end{proof}

Proposition \ref{prop:ar_congr}
makes the following definition
well-defined.

\begin{defn}{[Functor $\mathbf{L}$]}
\label{defn:fun_L}
A functor
$
\mathbf{L} : \Prob \to \Set
$
is defined by:
\begin{equation*}
\xymatrix@C=15 pt@R=30 pt{
    X
  &
    \ProbX
     \ar @{->}_{f^-} [d]
     \ar @{|->}^{\mathbf{L}} [rr]
  &&
    \mathbf{L}\ProbX
     \ar @{}^-{:=} @<-6pt> [r]
     \ar @{->}^{
       \mathbf{L} f^-
     } [d]
  &
    L^{\infty}(\ProbNX)
     \ar @{}^-{\ni} @<-6pt> [r]
  &
    [u]_{\sim_{\mathbb{P}_X}}
     \ar @{|->}^{
       \mathbf{L} f^-
     } @<-9pt> [d]
\\
    Y
      \ar @{->}^f [u]
  &
    \ProbY
     \ar @{|->}^{\mathbf{L}} [rr]
  &&
    \mathbf{L}\ProbY
     \ar @{}^-{:=} @<-6pt> [r]
  &
    L^{\infty}(\ProbNY)
     \ar @{}^-{\ni} @<-6pt> [r]
  &
    [u \circ f]_{\sim_{\mathbb{P}_Y}}
}
\end{equation*}
\end{defn}


\begin{exmp}
\label{exmp:tb1}

Let
$\omega$
be the category 
whose objects are all integers starting with 
$0$
and for each pair of
integers
$s$
and
$t$
with
$s \le t$
there is a unique arrow
$
*_{s, t}
	:
s
	\to
t
$.
That is,
$\omega$
is the category corresponding to the integer set
$\mathbb{N}$
with the usual total order.
For a real number
$p \in (0, 1)$,
we define a functor
$
\mathcal{B}
	:=
\mathcal{B}^p
	: \omega \to \Prob
$
in the following way.

For
an object
$t$
of
$\omega$,
$\mathcal{B}t$
is 
a probability space
$
\bar{X}_t :=
(X_t, \Sigma_t, \mathbb{P}_t)
$
whose components are defined as follows:
\begin{enumerate}
\item
$
X_t := \{
	0, 1
\}^t
$,
the set of all binary numbers of 
$t$
digits,

\item
$
\Sigma_t
	:=
2^{X_t}
$,

\item
for
$a \in X_t$,
$
\mathbb{P}_t
	:
\Sigma_t
	\to
[0,1]
$
is the probability measure defined by
$
\mathbb{P}_t(\{a\})
	:=
p^{\#a}
(1-p)^{t - \#a}
$
where
$\#a$
is the number of occurrences of $1$ in $a$.
\end{enumerate}

For  an integer
$t$,
$
F(
	*_{t, t+1}
)
	:=
f_t
$
is defined by
$
f_t(i_0 i_1 \dots i_t i_{t+1})
	:=
i_0 i_1 \dots i_t
$
where
$i_k$
is
$0$
or
$1$.
For 
$s < t$,
We write
$
f_{s, t}
$
for
$
F(*_{s, t})
	=
f_s
	\circ
f_{s+1}
	\circ
\dots
	\circ
f_{t-1}
$.

Since
$\Sigma_t$
is a powerset of 
$X_t$,
any function from 
$X_t$
is measurable.
Moreover
by the definition of
$\mathbb{P}_t$,
only null set in
$\Sigma_t$
is
$\emptyset$.
Therefore any function between
$X_s$
and
$X_t$
is null-preserving.
Hence,
$f_{s,t}$
is a 
$\Prob$-arrow.
Thus,
the functor
$\mathcal{B}$
is well-defined.
\end{exmp}

The functor
$\mathcal{B}$
represents a filtration over
the classical binomial model,
for example developed in
\cite{shreve_I}.
So we can think 
$\mathcal{B}$
a sort of
\newword{generalized filtration}.

One of the biggest difference between the classical and
$\Prob$ versions of binomial models
is
that the classical version requires the terminal time horizon
$T$
for determining the underlying set
$
\Omega :=
	\{0, 1\}^T
$
while our version does not require it
since the time variant probability spaces can evolve without any limit.
That is, our version allows unknown future elementary events,
which, we believe, shows a big philosophical difference from the Kolmogorov world.


\section{Conditional Expectation Functor}
\label{sec:CE_fun}

\if \obsolete 1
{\color{red}
\begin{thm}
\label{thm:CEf}
Let
$
f^-
$
be an arrow in
$
\Prob(
  \ProbX,
  \ProbY
)
$.
Then, for any
$
v \in
\mathcal{L}^{1}(\ProbNY)
$
there exists a
$
u \in
\mathcal{L}^{1}(\ProbNX)
$
such that
for every
$
A \in \Sigma_X
$
\begin{equation}
\label{eq:RN}
\int_A u \, d \mathbb{P}_X
  =
\int_{f^{-1}(A)} v \, d \mathbb{P}_Y.
\end{equation}
Moreover,
$u$
is determined uniquely up to 
$\mathbb{P}_X$-null sets.
In other words,
if there are two
$
u_1, u_2 \in
\mathcal{L}^{1}(\ProbNX)
$
both satisfying
(\ref{eq:RN}),
then
$
\mathbb{P}_X(u_1 \ne u_2)
  = 0
$
or
$
u_1
  \sim_{\mathbb{P}_X}
u_2
$.

We write a version of this 
$u$
by
$
E^{f^-}(v)
$,
and call it a
\newword{conditional expectation of $v$ along $f^-$}.
Therefore,
\begin{equation}
\int_A 
  E^{f^-}(v)
\, d \mathbb{P}_X
  =
\int_{f^{-1}(A)} v \, d \mathbb{P}_Y.
\end{equation}
\end{thm}
\begin{proof}
Let us prove the uniqueness first.
Assume that both
$u_1$
and
$u_2
  \in
\mathcal{L}^{1}(\ProbNX)
$
satisfy
(\ref{eq:RN}),
and that
$
\mathbb{P}_X(u_1 \ne u_2) > 0
$.
Then
$
\int_A u_1 \, d \mathbb{P}_X
   =
\int_A u_2 \, d \mathbb{P}_X
$.
Now let
$
w
  :=
u_1 - u_2
$.
Then
$
\mathbb{P}_X(w > 0)
  +
\mathbb{P}_X(w < 0)
  =
\mathbb{P}_X(w \ne 0)
  > 0
$.
Therefore
$
\mathbb{P}_X(w > 0)
  > 0
$
or
$
\mathbb{P}_X(w < 0)
  > 0
$.
In the former case,
there exists
$
n \in \mathbb{N}
$
such that
$
\mathbb{P}_X(w > \frac{1}{n})
  > 0
$
since
$
\{ w > 0 \}
  =
\uparrow\lim \{ w > \frac{1}{n} \}
$
and the monotone convergence property of measures.
But, then
$
0
  =
\int_{\{ w > \frac{1}{n}\}}
  w \,
d \mathbb{P}_X
  \ge
\frac{1}{n}
\mathbb{P}_X(w > \frac{1}{n})
  > 0
$,
which is a contradiction.
Hence
$
\mathbb{P}_X(w > 0) = 0
$.
A similar discussion works for the latter case.
Therefore,
$
\mathbb{P}_X(u_1 \ne u_2) = 0
$.

Now,
for
$
v \in
\mathcal{L}^{1}(\ProbNY)
$,
define a function
$
v^*
  :
\Sigma_Y \to \mathbb{R}
$
by
$
v^*(B)
  :=
\int_B v \, d \mathbb{P}_Y
$
for
$
B
  \in
\Sigma_Y
$.
Then,
$
v^*
$
is a signed measure
satisfying
$
v^* \ll \mathbb{P}_Y
$.
Therefore,
$
v^* \circ f^{-1} 
  \ll
\mathbb{P}_Y \circ f^{-1}
  \ll
\mathbb{P}_X
$.
So by Radon-Nikodym theorem,
there exists 
$
u \in
\mathcal{L}^{1}(\ProbNX)
$
such that
$
\int_A 
  u \,
d \mathbb{P}_X
  =
\int_A \,
d (v^* \circ f^{-1})
$
for any
$
A \in \Sigma_X
$,
where
$u$
is unique
up to 
$\mathbb{P}_X$-null sets.
Hence
\begin{equation*}
\int_A 
  u \,
d \mathbb{P}_X
  =
\int_A \,
d (v^* \circ f^{-1})
  =
(v^* \circ f^{-1})(A)
  =
v^* ( f^{-1}(A))
  =
\int_{f^{-1}(A)}
  v
\, d \mathbb{P}_Y .
\end{equation*}

\end{proof}} 
\fi 
\begin{defn}
Let us consider a $\Prob$ arrow $f^- : \ProbX \to \ProbY$. Take $v \in \mathcal{L}^1(\ProbNY)$ and put
\[
 v^*(B) := \int_B v \, d\mathbb{P}_Y
\]
for $B \in \Sigma_Y$. Then $v^* \circ f^{-1}$ is absolutely continuous w.r.t. $\mathbb{P}_X$, since $f^{-1}$ maps $\mathbb{P}_X$-null sets to $\mathbb{P}_Y$-null sets and
\[
 v^* \circ f^{-1}(A) = \int_{f^{-1}(A)} v \, d\mathbb{P}_Y \ (A \in \Sigma_X).
\]
So, thanks to Radon-Nikodym theorem, we have the unique (up to $\mathbb{P}_X$-a.s.) element $E^{f^-}(v)$ of $\mathcal{L}^1(\ProbNX)$ such that
\begin{equation}
 \label{eq:RN_E}
 \int_A E^{f^-}(v) \, d\mathbb{P}_X = \int_{f^{-1}(A)} v \, d\mathbb{P}_Y
\end{equation}
for all $A \in \Sigma_X$. We call this element $E^{f^-}(v)$ \newword{the conditional expectation of $v$ along $f^-$}.
\end{defn}

\begin{prop}
\label{prop:E_id}
For
$
u \in 
\mathcal{L}^1(\ProbNX)
$,
$
E^{Id_X^-}(u)
  \sim_{\mathbb{P}_X}
u
$.
\end{prop}
\begin{proof}
For every 
$A \in \Sigma_X$,
$
\int_A
  E^{Id_X^-}(u)
\, d \mathbb{P}_X 
  =
\int_{Id_X^{-1}(A)}
  u
\, d \mathbb{P}_X 
  =
\int_{A}
  u
\, d \mathbb{P}_X 
$.
\end{proof}

\begin{prop}
\label{prop:fun_E}
Let
$f^-$
and
$g^-$
be arrows in 
$\Prob$
like:
\begin{equation*}
\xymatrix{
  \ProbX
       \ar @{->}^{f^-} [r]
 &
  \ProbY
       \ar @{->}^{g^-} [r]
 &
  \ProbZ
}.
\end{equation*}
\begin{enumerate}
\item
For
$
v_1, v_2
  \in
\mathcal{L}^{1}(\ProbNY)
$,
$
v_1
  \sim_{\mathbb{P}_Y}
v_2
$
implies
$
E^{f^-}(v_1)
  \sim_{\mathbb{P}_X}
E^{f^-}(v_2)
$.

\item
For
$
w
  \in
\mathcal{L}^{1}(\ProbNZ)
$,
$
E^{f^-}(E^{g^-}(w))
  \sim_{\mathbb{P}_X}
E^{g^- \circ f^-}(w)
$.
\end{enumerate}
\end{prop}
\begin{proof}
\begin{enumerate}
\item
Assume that
$
v_1
  \sim_{\mathbb{P}_Y}
v_2
$.
Then, it is obvious that
$
v_1^*
  =
v_2^*
$
as functions.
The result comes from the uniqueness (up to 
$\mathbb{P}_X$-null sets)
of conditional expectations.

\item
It is sufficient to show that
for every
$
A \in \Sigma_X
$
\begin{equation*}
\int_A
  E^{f^-}(E^{g^-}(w))
\, d \mathbb{P}_X
  =
\int_{(f \circ g)^{-1}(A)}
  w \,
d \mathbb{P}_Z .
\end{equation*}
However, 
we can get this immediately by
applying 
(\ref{eq:RN_E}) twice.

\end{enumerate}
\end{proof}

Proposition \ref{prop:E_id}
and
Proposition \ref{prop:fun_E}
make the following definition
well-defined.

\begin{defn}{[Functor $\mathcal{E}$]}
\label{defn:fun_E}
A functor
$
\mathcal{E} : \Prob^{op} \to \Set
$
is defined by:
\begin{equation*}
\xymatrix@C=15 pt@R=30 pt{
    X
  &
    \ProbX
     \ar @{->}_{f^-} [d]
     \ar @{|->}^{\mathcal{E}} [rr]
  &&
    \mathcal{E}\ProbX
     \ar @{}^-{:=} @<-6pt> [r]
  &
    L^{1}(\ProbNX)
     \ar @{}^-{\ni} @<-6pt> [r]
  &
    [E^{f^-}(v)]_{\sim_{\mathbb{P}_X}}
\\
    Y
      \ar @{->}^f [u]
  &
    \ProbY
     \ar @{|->}^{\mathcal{E}} [rr]
  &&
    \mathcal{E}\ProbY
     \ar @{}^-{:=} @<-6pt> [r]
     \ar @{->}_{
       \mathcal{E} f^-
     } [u]
  &
    L^{1}(\ProbNY)
     \ar @{}^-{\ni} @<-6pt> [r]
  &
    [v]_{\sim_{\mathbb{P}_Y}}
     \ar @{|->}_{
       \mathcal{E} f^-
     } @<9pt> [u]
}
\end{equation*}
We call
$\mathcal{E}$
a 
\newword{conditional expectation functor}.
\end{defn}

Note that
the functors
$L$
and
$\mathcal{E}$
defined in
\cite{adachi_2014crm}
from the category
$
\chi(\Omega, \mathcal{G})
$
to
$\Set$
are representable as
$\mathbf{L} \circ \iota$
and
$\mathcal{E} \circ \iota$,
respectively
by using 
$\mathbf{L}$
and
$\mathcal{E}$
defined in this paper.
That is,
$\Prob$
is a more general and richer
category
than
$\chi$,
while
still having enough structure to define
conditional expectation functor.

One may wonder why we do not use more structured category 
such as the category of Banach spaces instead of using 
$\Set$.
One of our hidden goals when we defined the functors 
$\mathbf{L}$
and
$\mathcal{E}$
is
to develop a model of a logical inference system based on 
$\Prob$.
In order to make it possible,
we wanted to make the functor category over
$\Prob$
be a topos.
Picking 
$\Set$
as a target category is
a natural consequence of this line
since
the functor category
$
\Set^{\Prob}
$
becomes a topos.


The following three propositions
state
basic properties of
our conditional expectations,
which are similar to those of
classical conditional expectations.


\begin{prop}{[Linearity]}
\label{prop:lin}
Let
$
f^-
  :
\ProbX
  \to
\ProbY
$
be a 
$\Prob$-arrow.
Then
for every pair of 
random variables
$
u, v
  \in
 \mathcal{L}^1(\ProbNY)
$
and
$
\alpha ,\beta
\in \mathbf{R}
$,
we have
\begin{equation}
\label{eq:lin}
E^{f^-} (
    \alpha u + \beta v
)
   \sim_{\mathbb{P}_X}
\alpha E^{f^-}(u) + \beta E^{f^-}(v) .
\end{equation}
\end{prop}
\begin{proof}
For all 
$A \in \Sigma_X$,
\begin{align*}
\int_A E^{f^-}
  ( \alpha u + \beta v )
\, d \mathbb{P}_X
    &=
\int_{f^{-1}(A)}
  ( \alpha u + \beta v ) 
\, d \mathbb{P}_Y
    \\&=
\alpha
\int_{f^{-1} (A)}
   u
\, d \mathbb{P}_Y
   +
\beta
\int_{f^{-1}(A)}
   v
\, d \mathbb{P}_Y
    \\&=
\alpha
\int_A
   E^{f^-} (u)
 \, d \mathbb{P}_X
   +
\beta
\int_A
   E^{f^-}(v)
\, d \mathbb{P}_X
    \\&=
\int_A
   (\alpha E^{f^-}(u) + \beta E^{f^-}(v))
\, d \mathbb{P}_X .
\qedhere
\end{align*}
\end{proof}

\begin{prop}{[Positivity]}
\label{prop:pos}
Let
$
f^-
  :
\ProbX
  \to
\ProbY
$
be a 
$\Prob$-arrow.
If 
a random variable
$
v \in \mathcal{L}^1(\ProbNY)
$
is
$\mathbb{P}_Y$-almost surely positive,
i.e. 
$
v \ge 0
\  (\mathbb{P}_Y \text{-a.s.})
$,
then
$
E^{f^-} (v)
  \ge 0
\  (\mathbb{P}_X \text{-a.s.})
$.
\end{prop}
\begin{proof}
Since 
$v$
is
$\mathbb{P}_Y$-almost surely positive,
$
v^* \circ f^{-1}(A)
  =
\int_{f^{-1}(A)}
   v
\, d \mathbb{P}_Y
$
is a measure on
$(X, \Sigma_X)$
for every
$
A \in \Sigma_X
$.
Thus
$
E^{f^-}(v) \ge 0 
\;
(\mathbb{P}_X \text{-a.s.})
$
because
$
E^{f^-}(v)
$
is a Radon-Nikodym derivative
$
d (v^* \circ f^{-1})
 /
d \mathbb{P}_X
$.
\end{proof}

\begin{prop}{[Monotone Convergence]}
\label{prop:mc}
Let
$
f^-
  :
\ProbX
  \to
\ProbY
$
be a 
$\Prob$-arrow,
$
v,
v_n
   \in
\mathcal{L}^1(\ProbNY)
$
be random variables
for
$n \in \mathbf{N}$.
If
$
0
  \le
v_n
  \uparrow
v
\ (\mathbb{P}_Y \text{-a.s.})
$,
then
$
0
   \le
E^{f^-}(v_n)
   \uparrow
E^{f^-}(v)
\ (\mathbb{P}_X \text{-a.s.})
$.
\end{prop}
\begin{proof}
By 
Proposition \ref{prop:lin}
 and
Proposition \ref{prop:pos},
we have
$
0
   \le
E^{f^-}(v_n)
   \le
E^{f^-}(v_{n+1})
\ (\mathbb{P}_X \text{-a.s.})
$
for all
$n \in \mathbf{N}$.
Put
$
h
  :=
\limsup_n E^{f^-}(v_n)
$,
then obviously,
$
E^{f^-}(v_n)
   \uparrow
h
\ (\mathbb{P}_X \text{-a.s.}).
$
So all we need to show is that
for every
$A \in \Sigma_X$,
$
\int_A
   E^{f^-}(v)
\, d \mathbb{P}_X
   =
\int_A
   h
\, d \mathbb{P}_X .
$
But, thanks to monotone convergence theorem,
we have
\begin{align*}
\int_A
   h
\, d \mathbb{P}_X
   &=
\lim_{n \rightarrow \infty}
   \int_A
      E^{f^-}(v_n)
   \, d \mathbb{P}_X
   =
\lim_{n \rightarrow \infty}
   \int_{f^{-1}(A)}
      v_n
   \, d \mathbb{P}_Y
    \\&=
\int_{f^{-1}(A)}
    v
\, d \mathbb{P}_Y
    =
\int_A
   E^{f^-}(v)
\, d \mathbb{P}_X.
\qedhere
\end{align*}
\end{proof}

\begin{defn}{[Unconditional Expectation]}
\label{defn:uncondExp}
For
$
v
  \in
\mathcal{L}^{1}(\ProbNY),
$
we call
$
E^{!_Y^-}(v)
$
a 
\newword{unconditional expectation}
of
$v$,
where
$
!_Y^-
$
is the unique arrow
$
!_Y^-
  :
\mathbb{0}
  \to
\ProbY
$ defined in
Proposition \ref{prop:init_obj}.
\end{defn}

\begin{prop}
\label{prop:uncond_exp}
Let 
$
!_Y^-
  :
\mathbb{0}
  \to
\ProbY
$
be the
unique
$\Prob$-arrow
and
$
v
  \in
\mathcal{L}^{1}(\ProbNY)
$.
Then, we have
\begin{equation}
\label{eq:uncond_val}
E^{!_Y^-}(v)(*)
  =
\mathbb{E}^{\mathbb{P}_Y}[v].
\end{equation}
\end{prop}
\begin{proof}
\begin{equation*}
E^{!_Y^-}(v)(*)
  =
\int_{\{*\}}
  E^{!_Y^-}(v)
\, d \mathbb{P}_{\mathbb{0}}
  =
\int_{
  !_Y^{-1}
  (\{*\})
}
  v
\, d \mathbb{P}_Y
  =
\int_Y
  v
\, d \mathbb{P}_Y
  =
\mathbb{E}^{\mathbb{P}_Y}[v].
\end{equation*}
\end{proof}

Proposition \ref{prop:uncond_exp}
asserts that 
our unconditional expectation is a 
natural extension of the
classical one.

\section{$f^-$-measurability}
\label{sec:meas}

\begin{defn}{[$f^-$-measurability]}
\label{defn:measu}
Let 
$
f^-
  :
\ProbX
  \to
\ProbY
$
be a
$\Prob$-arrow
and
$
v
  \in
\mathcal{L}^{\infty}(\ProbNY)
$.
$v$
is called
\newword{$f^-$-measurable}
if
there exists
$
w
  \in
\mathcal{L}^{\infty}(\ProbNX)
$
such that
$
v
  \sim_{\mathbb{P}_Y}
w \circ f
$.

\end{defn}
The following proposition allows us to say that
an element of
$
L\ProbY
$
is
$
f^-
$-measurable.

\begin{prop}
\label{prop:f_measu_invariance}
Let 
$
f^-
  :
\ProbX
  \to
\ProbY
$
be a
$\Prob$-arrow
and
$
v_1
$
and
$
v_2
$
be two elements of
$
\mathcal{L}^{\infty}(\ProbNY)
$
satisfying
$
v_1
   \sim_{\mathbb{P}_Y}
v_2
$.
Then,
if
$v_1$
is $f^-$-measurable,
so is 
$v_2$.
\end{prop}
\begin{proof}
Obvious.
\end{proof}


The next proposition is well-known.
For example, see
Page 206 of 
\cite{williams1991}.

\begin{prop}
\label{prop:f_meas_classic}
Let 
$
f^-
  :
\ProbX
  \to
\ProbY
$
be a
$\Prob$-arrow
and
$
v
  \in
\mathcal{L}^{\infty}(\ProbNY)
$.
Then,
$v$
is
$f^-$-measurable
if and only if
$v$
is
$
f^{-1}(\Sigma_X)
  /
\mathcal{B}(\mathbb{R})
$-measurable.

\end{prop}
Proposition \ref{prop:f_meas_classic}
says that
$f^-$-measurability
is 
 an extension of the classical measurability.

\begin{thm}
\label{thm:measu}
Let 
$
f^-
  :
\ProbX
  \to
\ProbY
$
be a
$\Prob$-arrow,
$ u $
be an element of
$\mathcal{L}^{1}(\ProbNY)$
and
$ v $
be a random variable in
$
\mathcal{L}^{\infty}(\ProbNY)
$,
and assume that
$v$
is
$f^-$-measurable.
Then
we have
\begin{equation}
E^{f^-}(
  v \cdot u
)
  \sim_{\mathbb{P}_X}
w 
  \cdot
E^{f^-}(u) ,
\end{equation}
where
$
w
  \in
\mathcal{L}^{\infty}(\ProbNX)
$
is a random variable satisfying
$
v
  \sim_{\mathbb{P}_Y}
w \circ f
$.

\end{thm}
\begin{proof}
By (\ref{eq:RN_E}),
it is sufficient to prove that
for every
$
A
  \in
\Sigma_X 
$,
\begin{equation}
\label{eq:f_meas_pr}
\int_{f^{-1}(A)}
   v \cdot u
\, d \mathbb{P}_Y
   =
\int_A
  w
\cdot
  E^{f^-}(u)
\, d \mathbb{P}_X.
\end{equation}

\if \obsolete 1
{\color{red}
Firstly, we prove 
(\ref{eq:f_meas_pr})
 in the case when
$
w
  =
1_F
$
with some
$
F
  \in
\Sigma_X
$.
In the following,
note that 
$
1_F \circ f
  =
1_{f^{-1}(F)}
$.
\begin{align*}
&
\int_{f^{-1}(A)}
   v \cdot u
\, d \mathbb{P}_Y
   =
\int_{f^{-1}(A)}
   1_{f^{-1}(F)}
   \cdot u
\, d \mathbb{P}_Y
   =
\int_{
  f^{-1}(A)
\cap
  f^{-1}(F)
}
   u
\, d \mathbb{P}_Y
\\&
   =
\int_{
  A \cap F
}
  E^{f^-}(u)
\, 
d \mathbb{P}_X
   =
\int_A
  1_F
\cdot
  E^{f^-}(u)
\, d \mathbb{P}_X
   =
\int_A
  w
\cdot
  E^{f^-}(u)
\, d \mathbb{P}_X.
\end{align*}

Secondly, we consider
the case when
$
w
  =
\sum_{i=1}^n
  a_i
  1_{F_i}
$
for
$
n \in \mathbb{N}
$,
$
a_i \in \mathbb{R}
$
and
$
F_i
  \in
\Sigma_X
$.
Note that
$
v =
\sum_{i=1}^n
  a_i
  1_{
    f^{-1}(F_i)
  }
$.
Then,
by the result of the first case,
\begin{align*}
\int_{f^{-1}(A)}
   v \cdot u
\, d \mathbb{P}_Y
   &=
\sum_{i=1}^n
  a_i
\int_{f^{-1}(A)}
   1_{
     f^{-1}(F_i)
   }
   \cdot u
\, d \mathbb{P}_Y
  \\&=
\sum_{i=1}^n
  a_i
\int_A
   1_{F_i}
     \cdot 
   E^{f^-}(u)
\, d \mathbb{P}_X
   =
\int_A
   w
     \cdot 
   E^{f^-}(u)
\, d \mathbb{P}_X.
\end{align*}

Finally,
in the case when
$
w
   \in
\mathcal{L}^{\infty}(\ProbNX)
$,
there exists a sequence of step functions
$
\{ w_n\}_{n \in \mathbb{N}}
$
such that
$
\abs{w_n}
  \le
\abs{w}
$
and
$
w_n
  \rightarrow
w
$.
Since
$
w_n \circ f
  \rightarrow
v
$,
we have 
\begin{align*}
\int_{f^{-1}(A)}
   v \cdot u
\, d \mathbb{P}_Y
   &=
\lim_{n \rightarrow \infty}
  \int_{f^{-1}(A)}
     (w_n \circ f)
        \cdot
     u
  \, d \mathbb{P}_Y
   \\&=
\lim_{n \rightarrow \infty}
  \int_A
     w_n
        \cdot
     E^{f^-}(u)
  \, d \mathbb{P}_X
   =
\int_A
   w
     \cdot 
   E^{f^-}(u)
\, d \mathbb{P}_X
\end{align*}
by the result of the second case,
which completes the proof.} 
\fi 

But, it is obvious from the transformation theorem applying with the Jordan decomposition.
\end{proof}

Theorem \ref{thm:measu}
is a generalization of a classical formula
\begin{equation*}
\mathbb{E}^{\mathbb{P}}[
  v \cdot u
\mid
  \mathcal{G}
]
  \sim_{\mathbb{P}}
v
  \cdot
\mathbb{E}^{\mathbb{P}}[
  u
\mid
  \mathcal{G}
]
\end{equation*}
for a
$\mathcal{G}$-measurable 
random variable
$v$.

The following theorem has
some categorical taste.

\begin{thm}
\label{thm:catMeas}
Let
$
 \mathcal{E} \boxtimes \mathbf{L},
 \mathcal{E} \mathcal{P}_1
:
 \Prob^{op} \times \Prob
\to
 \Set
$
be two parallel bifunctors defined by
\begin{equation*}
\mathcal{E} \boxtimes \mathbf{L}
  :=
\Box
  \circ
(\mathcal{E} \times \mathbf{L})
\; \; \textrm{and} \; \;
\mathcal{E} \mathcal{P}_1
  :=
\mathcal{E}
  \circ
\mathcal{P}_1
\end{equation*}
where
$
\mathcal{P}_1
  :
\Prob^{op} \times \Prob
  \to
\Prob^{op}
$
is the projection for the first component,
and
$
\Box
  :
\Set \times \Set
  \to
\Set
$
is a functor 
which sending an ordered pair of sets to the set product of its components. 

Now, for each
$\Prob$-object
$\ProbX$,
define a function
$
\alpha_{\ProbX}
  :
L^{1}(\ProbNX) \times L^{\infty}(\ProbNX)
  \to
L^{1}(\ProbNX)
$
by
$
\alpha_{\ProbX}(
  \pair{
     [
        u
     ]_{\sim_{\mathbb{P}_X}}
   }{
     [
        v
     ]_{\sim_{\mathbb{P}_X}}
   }
)
  =
[
   u \cdot v
]_{\sim_{\mathbb{P}_X}}
$.
Then the following diagram commutes.
\begin{equation*}
\xymatrix@C=12 pt@R=10 pt{
    \Prob^{op} \times \Prob
      \ar @{->}^-{
        \mathcal{E} \boxtimes \mathbf{L}
      } [r]
&
    \Set
&&
    \Set
&
    \Prob^{op} \times \Prob
      \ar @{->}_-{
        \mathcal{E} \mathcal{P}_1
      } [l]
\\
    \pair{\ProbX}{\ProbX}
&
    L^{1}(\ProbNX) \times L^{\infty}(\ProbNX)
      \ar @{->}^-{\alpha_{\ProbX}} [rr]
&&
    L^{1}(\ProbNX)
      \ar @{->}^{
        \mathcal{E} Id_X^-
      } [dd]
&
    \pair{\ProbX}{\ProbX}
      \ar @{->}^{
        \pair{Id_X^-}{f^-}
      } [dd]
\\\\
    \pair{\ProbY}{\ProbX}
      \ar @{->}^{
        \pair{f^-}{Id_X^-}
      } [uu]
      \ar @{->}_{
        \pair{Id_Y^-}{f^-}
      } [dd]
&
    L^{1}(\ProbNY) \times L^{\infty}(\ProbNX)
      \ar @{->}^{
        \mathcal{E} f^- \times \mathbf{L} Id_{X}^-
      } [uu]
      \ar @{->}_{
        \mathcal{E} Id_Y^- \times \mathbf{L} f^-
      } [dd]
&&
    L^{1}(\ProbNX)
&
    \pair{\ProbX}{\ProbY}
\\\\
    \pair{\ProbY}{\ProbY}
&
    L^{1}(\ProbNY) \times L^{\infty}(\ProbNY)
      \ar @{->}^-{\alpha_{\ProbY}} [rr]
&&
    L^{1}(\ProbNY)
      \ar @{->}_{
        \mathcal{E} f^-
      } [uu]
&
    \pair{\ProbY}{\ProbY}
      \ar @{->}_{
        \pair{f^-}{Id_Y^-}
      } [uu]
}
\end{equation*}
In other words,
$
  \alpha :
    \mathcal{E} \boxtimes \mathbf{L}
      \xrightarrow{\centerdot \centerdot}
    \mathcal{E} \mathcal{P}_1
$
is a dinatural transformation.

\end{thm}
\begin{proof}
For
$
\pair{
  [v]_{\sim_{\mathbb{P}_Y}}
}{
  [u]_{\sim_{\mathbb{P}_X}}
}
  \in
L^{1}(\ProbNY) \times L^{\infty}(\ProbNX)
$,
we have
\begin{align*}
(
 \mathcal{E} Id_X^-
   \circ
 \alpha_{\ProbX}
   \circ
 (\mathcal{E} f^- \times \mathbf{L} Id_{X}^-)
)(
\pair{
  [v]_{\sim_{\mathbb{P}_Y}}
}{
  [u]_{\sim_{\mathbb{P}_X}}
}
)
  &=
\big[
  E^{f^-}(v) \cdot u
\big]_{\sim_{\mathbb{P}_X}},
\\
(
 \mathcal{E} f^-
   \circ
 \alpha_{\ProbY}
   \circ
 (\mathcal{E} Id_Y^- \times \mathbf{L} f^-)
)(
\pair{
  [v]_{\sim_{\mathbb{P}_Y}}
}{
  [u]_{\sim_{\mathbb{P}_X}}
}
)
  &=
\big[
  E^{f^-}(v \cdot (u \circ f))
\big]_{\sim_{\mathbb{P}_X}}
\end{align*}
since 
$
\mathcal{E} Id_X^-
  =
Id_{L^1(\ProbNX)}
$.
But by Theorem \ref{thm:measu},
two rightmost hand sides coincide,
which completes the proof.
\end{proof}

\section{$f^-$-independence}
\label{sec:indep}

\begin{defn}{[Category $\mpProb$]}
\label{defn:mp}
A 
$\Prob$-arrow
$
f^-
  :
\ProbX
  \to
\ProbY
$
is called
\newword{measure-preserving}
if
$
\mathbb{P}_Y
  \circ
f^{-1}
  =
\mathbb{P}_X
$.
A subcategory
$\mpProb$
of 
$\Prob$
is a category
whose objects are same as those of $\Prob$
but arrows are limited to all measure-preserving arrows.
\end{defn}

Franz defines stochastic independence
in the opposite category of $\mpProb$
as an example of his introducing notion of
stochastic independence in monoidal categories.

\begin{defn}{\cite{franz_2003}}
\label{defn:mpIndep}
Two 
$\mpProb$-arrows
$
f^- :
  \ProbX
\to
  \ProbZ
$
and
$
g^- :
  \ProbY
\to
  \ProbZ
$
are called
\newword{independent}
if there exists an
$\mpProb$-arrow
$
q^- :
  \ProbX \otimes \ProbY
\to
  \ProbZ
$
such that
 the following diagram commutes
\begin{equation*}
\xymatrix@C=20 pt@R=20 pt{
&
  \ProbZ
\\
  \ProbX
    \ar @{->}_-{p_1^-} [r]
    \ar @{->}^-{f^-} [ru]
&
  \ProbX \otimes \ProbY
    \ar @{.>}^-{q^-} [u]
&
  \ProbY
    \ar @{->}^-{p_2^-} [l]
    \ar @{->}_-{g^-} [lu]
}
\end{equation*}
where
$
\ProbX \otimes \ProbY
  :=
(
  X \times Y,
  \Sigma_X \otimes \Sigma_Y,
  \mathbb{P}_X \otimes \mathbb{P}_Y
)
$,
$p_1$
and
$p_2$
are projections,
$
  \Sigma_X \otimes \Sigma_Y
$
is the smallest
$\sigma$-algebra
of 
$
  X \times Y
$
making
both
$p_1$
and
$p_2$
measurable,
and
$
  \mathbb{P}_X \otimes \mathbb{P}_Y
$
is a product measure such that
$
(\mathbb{P}_X \otimes \mathbb{P}_Y)(A \times B)
  =
\mathbb{P}_X(A)
\mathbb{P}_Y(B)
$
for all
$
A \in \Sigma_X
$
and
$
B \in \Sigma_Y
$.
\end{defn}

Franz shows that
the notion of independence
defined in
Definition \ref{defn:mpIndep}
exactly matches the classical one in the sense of the following proposition.

\begin{prop}{\cite{franz_2003}}
\label{prop:mpIndep}
Two 
$\mpProb$-arrows
$
f^- :
  \ProbX
\to
  \ProbZ
$
and
$
g^- :
  \ProbY
\to
  \ProbZ
$
are independent
if and only if
for every pair of
$A \in \Sigma_X$
and
$B \in \Sigma_Y$
\begin{equation}
\label{eq:def_indep}
\mathbb{P}_Z(
  f^{-1}(A)
\cap
  g^{-1}(B)
)
 =
\mathbb{P}_Z(
  f^{-1}(A)
)
\mathbb{P}_Z(
  g^{-1}(B)
).
\end{equation}
\end{prop}

Before extending the notion of independence to the category
$\Prob$,
we need the following note.

\begin{prop}
\label{prop:splitProbA}
Let
$
f^- :
  \ProbX
\to
  \ProbY
$
be a $\Prob$-arrow.
We define a
 $\Prob$-object
$\ProbX_{f^-}$
by
\begin{equation}
\ProbX_{f^-}
  :=
(
  X, \Sigma_X,
  \mathbb{P}_Y \circ f^{-1}
).
\end{equation}
Then, the following diagram commutes in $\Prob$
\begin{equation*}
\xymatrix@C=20 pt@R=20 pt{
  \ProbX
    \ar @{->}^-{f^-} [r]
    \ar @{->}_-{id_X^{-}} [d]
&
  \ProbY
\\
  \ProbX_{f^-}
    \ar @{->}_-{f^{\sim}} [ru]
}
\end{equation*}
where
$
f^{\sim}
$
and
$
id_X^{-}
$
are corresponding 
$\Prob$-arrows
of
$
f :
  \ProbY
\to
  \ProbX_{f^-}
$
and
$
id_X :
  \ProbX_{f^-}
\to
  \ProbX
$,
respectively.
Moreover,
$
f^{\sim}
$
is measure-preserving.

\end{prop}
\begin{proof}
Obvious.
\end{proof}

\begin{defn}{[Independence in $\Prob$]}
\label{defn:indepProb}
Two 
$\Prob$-arrows
$
f^- :
  \ProbX
\to
  \ProbZ
$
and
$
g^- :
  \ProbY
\to
  \ProbZ
$
are called
\newword{independent}
if there exists a
measure-preserving arrow
$
q^- :
  \ProbX_{f^-} \otimes \ProbY_{g^-}
\to
  \ProbZ
$
such that
 the following diagram commutes.
\begin{equation*}
\xymatrix@C=20 pt@R=20 pt{
  \ProbX
    \ar @{->}^-{f^-} [r]
    \ar @{->}_-{id_X^{-}} [d]
&
  \ProbZ
&
  \ProbY
    \ar @{->}_-{g^-} [l]
    \ar @{->}^-{id_Y^{-}} [d]
\\
  \ProbX_{f^-}
    \ar @{->}_-{p_1^-} [r]
    \ar @{->}^-{f^{\sim}} [ru]
&
  \ProbX_{f^-} \otimes \ProbY_{g^-}
    \ar @{.>}^-{q^-} [u]
&
  \ProbY_{g^-}
    \ar @{->}^-{p_2^-} [l]
    \ar @{->}_-{g^{\sim}} [lu]
}
\end{equation*}

\end{defn}

\begin{lem}
\label{lem:ELid}
For a measure-preserving
$\Prob$-arrow
$
f^- :
  \ProbX
\to
  \ProbY
$,
$
\mathcal{E} f^-
  \circ
\mathbf{L} f^-
  =
id_{\mathbf{L}\ProbX}
$.
\end{lem}
\begin{proof}
For
$
u \in
\mathcal{L}^{\infty}(\ProbNX)
$,
$
(\mathbf{L} f^-)[
  u
]_{\sim_{\mathbb{P}_X}}
  =
[
  u \circ f
]_{\sim_{\mathbb{P}_Y}}
$
is 
$f^-$-measurable.
Hence by
Theorem \ref{thm:measu},
\begin{equation*}
\mathcal{E} f^-(
  \mathbf{L} f^-(
    [
      u
    ]_{\sim_{\mathbb{P}_X}}
  )
)
  =
\big[
  E^{f^-}(
    u \circ f
  )
\big]_{\sim_{\mathbb{P}_X}}
  =
\big[
  u \cdot
  E^{f^-}(
    1_Y
  )
\big]_{\sim_{\mathbb{P}_X}}.
\end{equation*}
But, since $f^-$ is measure-preserving,
for all
$A \in \Sigma_X$
\begin{equation*}
\int_A 
  E^{f^-}(1_Y)
\, d \mathbb{P}_X
  =
\int_{f^{-1}(A)} 
  1_Y
\, d \mathbb{P}_Y
  =
\mathbb{P}_Y(f^{-1}(A))
  =
\mathbb{P}_X(A)
  =
\int_{A} 
  1_X
\, d \mathbb{P}_X.
\end{equation*}
Therefore,
$
  E^{f^-}(1_Y)
=
  1_X
$,
which concludes the proof.
\end{proof}

\begin{lem}
\label{lem:ELpi}
Let
$\ProbX$
and
$\ProbY$
be probability spaces.
Then
for all
$
v \in
\mathcal{L}^{\infty}(\ProbNY)
$,
\begin{equation}
E^{p_1^-}(v \circ p_2)
  \sim_{\mathbb{P}_X}
\mathbb{E}^{\mathbb{P}_{Y}}[v]
1_X
\end{equation}
where
$p_1$
and
$p_2$
are projections from
$X \times Y$
to
$X$ and $Y$, respectively.
\end{lem}
\begin{proof}
For
$
A \in \Sigma_X
$,
\begin{align*}
\int_A
  E^{p_1^-}(v \circ p_2)
\, d \mathbb{P}_X
  &=
\int_{p_1^{-1}(A)}
  v \circ p_2
\, d (\mathbb{P}_X \otimes \mathbb{P}_Y)
  \\&=
\int_{X \times Y}
  (v \circ p_2)
\cdot
  (1_A \circ p_1)
\, d (\mathbb{P}_X \otimes \mathbb{P}_Y)
  \\&=
\int_{Y}
\int_{X}
  ((v \circ p_2) \pair{x}{y})
\cdot
  ((1_A \circ p_1) \pair{x}{y})
\, \mathbb{P}_X(dx) 
\, \mathbb{P}_Y(dy)
  \\&=
\int_{Y}
  v(y)
\, \mathbb{P}_Y(dy)
\int_{X}
  1_A(x)
\, \mathbb{P}_X(dx) 
  \\&=
\mathbb{E}^{\mathbb{P}_Y}[
  v
]
\int_{A}
  1_X
\, d \mathbb{P}_X .
\end{align*}
\end{proof}

\begin{thm}
\label{thm:indep}
Let
$
f^- :
  \ProbX
\to
  \ProbZ
$
and
$
g^- :
  \ProbY
\to
  \ProbZ
$
be two independent
$\Prob$-arrows.
Then, 
for every
$
v \in \mathcal{L}^{\infty}(\ProbY)
$,
we have
\begin{equation}
E^{f^-}(v \circ g)
  \sim_{\mathbb{P}_X}
\mathbb{E}^{\mathbb{P}_Z}[
  v \circ g
]
E^{f^-}(1_Z).
\end{equation}

\end{thm}
\begin{proof}
For the diagram in Definition \ref{defn:indepProb},
apply $\mathcal{E}$ to its left box,
and
apply $\mathbf{L}$ to its right box.
Then, we get the following diagram.
\begin{equation*}
\xymatrix@C=20 pt@R=20 pt{
  L^1 \ProbX
&
  L^1 \ProbZ
    \ar @{->}_-{\mathcal{E} f^-} [l]
    \ar @{->}_-{\mathcal{E} f^{\sim}} [ld]
    \ar @{->}^-{\mathcal{E} q^-} [d]
&
  L^{\infty} \ProbZ
    \ar @{_{(}->} [l]
&
  L^{\infty} \ProbY
    \ar @{->}_-{\mathbf{L} g^-} [l]
    \ar @{->}^-{\mathbf{L} id_Y^{-}} [d]
\\
  L^1 \ProbX_{f^-}
    \ar @{->}^-{\mathcal{E} id_X^{-}} [u]
&
  L^1(\ProbX_{f^-} \otimes \ProbY_{g^-})
    \ar @{->}^-{\mathcal{E} p_1^-} [l]
&
  L^{\infty}(\ProbX_{f^-} \otimes \ProbY_{g^-})
    \ar @{->}^-{\mathbf{L} q^-} [u]
    \ar @{_{(}->} [l]
&
  L^{\infty} \ProbY_{g^-}
    \ar @{->}^-{\mathbf{L} p_2^-} [l]
    \ar @{->}_-{\mathbf{L} g^{\sim}} [lu]
}
\end{equation*}
The left and right boxes in the above diagram commute
since they are images of functors
$\mathcal{E}$
and
$\mathbf{L}$,
the center box also commutes by
Lemma \ref{lem:ELid},
and so does the whole diagram.
Now for
$
v \in \mathcal{L}^{\infty}(\ProbNY)
$,
let us see the values at the upper-left corner of the diagram
developed
through two paths,
which should coincide.
\begin{align*}
(
  \mathcal{E} f^-
\circ
  \mathbf{L} g^-
)
[
  v
]_{\sim_{\mathbb{P}_Y}}
  &=
[
  E^{f^-}(v \circ g)
]_{\sim_{\mathbb{P}_X}} ,
\\
(
  \mathcal{E} id_X^{-}
\circ
  \mathcal{E} p_1^-
\circ
  \mathbf{L} p_2^-
\circ
  \mathbf{L} id_Y^{-}
)
[
  v
]_{\sim_{\mathbb{P}_Y}}
  &=
[
  (
    E^{id_X^{-}}
      \circ
    E^{p_1^-}
  )
  (v \circ p_2)
]_{\sim_{\mathbb{P}_X}} .
\end{align*}
Hence by
Lemma \ref{lem:ELpi},
\begin{equation*}
  E^{f^-}(v \circ g)
\sim_{\mathbb{P}_X}
  \mathbb{E}^{
    \mathbb{P}_Z
      \circ
    g^{-1}
  }[
     v
  ]
  E^{id_X^{-}}(
     1_X
  )
\sim_{\mathbb{P}_X}
  \mathbb{E}^{
    \mathbb{P}_Z
  }[
     v \circ g
  ]
  E^{id_X^{-}}(
     1_X
  ) .
\end{equation*}
Now for every
$
A \in \Sigma_X
$,
\begin{align*}
\int_A
  E^{id_X^{-}}(
     1_X
  )
\, d \mathbb{P}_X
  &=
\int_{A}
  1_X
\, d (\mathbb{P}_Z \circ f^{-1})
  =
(\mathbb{P}_Z \circ f^{-1})
  (A)
  \\&=
\int_{f^{-1}(A)}
  1_Z
\, d \mathbb{P}_Z
  =
\int_A
  E^{f^-}(1_Z)
\, d \mathbb{P}_X .
\end{align*}
Therefore,
$
  E^{id_X^{-}}(
     1_X
  )
\sim_{\mathbb{P}_X}
  E^{f^-}(1_Z)
$,
which completes the proof.
\end{proof}

\begin{defn}{[$f^-$-independence]}
\label{defn:indep}
For a random variable
$
v
  \in
\mathcal{L}^{1}(\ProbNY)
$,
we define
a probability space
$
\bar{\mathbb{R}}_v
$
by
$
\bar{\mathbb{R}}_v
  :=
(
  \mathbb{R},
  \mathcal{B}(\mathbb{R}),
  \mathbb{P}_Y \circ v^{-1}
)
$.
Then,
$
v^-
  :
\bar{\mathbb{R}}_v
  \to
\ProbY
$
is a 
$\Prob$-arrow.
Now
for a
$\Prob$-arrow
$
f^-
  :
\ProbX
  \to
\ProbY
$,
$v$
is said to be
\newword{independent}
of
$f^-$,
denoted by
$
v \bot f^-
$,
if
$f^-$
and
$v^-$
are independent.
\end{defn}

The following proposition allows us to say that
an element of
$
L^{1}(\ProbNY)
$
is
independent of
$
f^-
$.

\begin{prop}
\label{prop:indep_invariance}
Let 
$
f^-
  :
\ProbX
  \to
\ProbY
$
be a
$\Prob$-arrow
and
$
v_1
$
and
$
v_2
$
be two elements of
$
\mathcal{L}^{1}(\ProbNY)
$
satisfying
$
v_1
   \sim_{\mathbb{P}_Y}
v_2
$.
Then,
$
v_1 \bot f^-
$
implies
$
v_2 \bot f^-
$.
\end{prop}
\begin{proof}
Assume that
$
v_1
   \sim_{\mathbb{P}_Y}
v_2
$
and
$
v_1 \bot f^-
$.
Let
$
N := \{
  y \in Y
\mid
  v_1(y) \ne v_2(y)
\}
$
and
$
M := Y - N
$.
Then,
$
\mathbb{P}_Y(N) = 0
$.

First,
we show that
\begin{equation}
\label{eq:indep_invariancev}
\mathbb{P}_Y \circ v_1^{-1}
  =
\mathbb{P}_Y \circ v_2^{-1} .
\end{equation}
For every
$
B \in
\mathcal{B}(\mathbb{R})
$
and
$
i = 1, 2
$,
\begin{equation*}
(\mathbb{P}_Y \circ v_i^{-1})(B)
  =
\mathbb{P}_Y(v_i^{-1}(B) \cap N)
  +
\mathbb{P}_Y(v_i^{-1}(B) \cap M)
  =
\mathbb{P}_Y(v_i^{-1}(B) \cap M) .
\end{equation*}
But,
\begin{equation*}
y \in v_1^{-1}(B) \cap M
  \Leftrightarrow
v_1(y) = v_2(y) \in B
  \Leftrightarrow
y \in v_2^{-1}(B) \cap M ,
\end{equation*}
which proves 
(\ref{eq:indep_invariancev}).
Hence,
$
\bar{\mathbb{R}}_{v_1}
  =
\bar{\mathbb{R}}_{v_2}
$.

Now since 
$
v_1 \bot f^-
$,
we have the following 
measure-preserving
$
q_1^-
$.
\begin{equation*}
\xymatrix@C=20 pt@R=20 pt{
  \ProbX
    \ar @{->}^-{f^-} [r]
    \ar @{->}_-{id_X^{-}} [d]
&
  \ProbY
\\
  \ProbX_{f^-}
    \ar @{->}_-{p_1^-} [r]
    \ar @{->}^-{f^{\sim}} [ru]
&
  \ProbX_{f^-}
     \otimes 
  \bar{\mathbb{R}}_{v_1}
    \ar @{.>}^-{q_1^-} [u]
&
\bar{\mathbb{R}}_{v_1}
    \ar @{->}^-{p_2^-} [l]
    \ar @{->}_-{v_1^-} [lu]
}
\end{equation*}
Then,
$
q_1
$
satisfies that
$
q_1(y)
  =
\pair{
  f(y)
}{
  v_1(y)
}
$
for all 
$
y \in Y
$.
Similarly we define a function
$
q_2
  :
Y
  \to
X \times R
$
by
$
q_2(y)
  =
\pair{
  f(y)
}{
  v_2(y)
}
$
for all 
$
y \in Y
$.
Then, all we need to show is that
$q_2^-$
is a measure-preserving
$\Prob$-arrow,
in other words,
\begin{equation}
\mathbb{P}_Y
  \circ
q_2^{-1}
  =
(\mathbb{P}_Y \circ f^{-1})
  \otimes
(\mathbb{P}_Y \circ v_2^{-1}) .
\end{equation}
However,
by
(\ref{eq:indep_invariancev})
and
the fact that
$q_1^-$
is measure-preserving,
 it is enough to show that
\begin{equation}
\label{eq:indep_invarianceq}
\mathbb{P}_Y \circ q_1^{-1}
  =
\mathbb{P}_Y \circ q_2^{-1} .
\end{equation}
For any
$
E
  \in
\Sigma_X \otimes \mathcal{B}(\mathbb{R})
$
and
$
i = 1, 2
$,
\begin{equation*}
(\mathbb{P}_Y \circ q_i^{-1})(E)
  =
\mathbb{P}_Y(q_i^{-1}(E) \cap N)
  +
\mathbb{P}_Y(q_i^{-1}(E) \cap M)
  =
\mathbb{P}_Y(q_i^{-1}(E) \cap M) .
\end{equation*}
But,
\begin{align*}
y \in
q_1^{-1}(E) \cap M
   &\Leftrightarrow
(
\pair{f(y)}{v_1(y)} \in E
  \, \land \,
v_1(y) = v_2(y)
)
\\
   &\Leftrightarrow
(
\pair{f(y)}{v_2(y)} \in E
  \, \land \,
v_1(y) = v_2(y)
)
   \Leftrightarrow
y \in
q_2^{-1}(E) \cap M ,
\end{align*}
which proves
(\ref{eq:indep_invarianceq}).
\end{proof}

\begin{prop}
\label{prop:0_indep}
Let 
$
!_Y^-
  :
\mathbb{0}
  \to
\ProbY
$
be a
unique
$\Prob$-arrow
and
$
v
  \in
\mathcal{L}^{1}(\ProbNY)
$.
Then, 
$v$
is independent of
$
!_Y^-
$.
\end{prop}
\begin{proof}
Obvious.
\end{proof}


\begin{thm}
\label{thm:f_indep}
Let 
$
f^-
  :
\ProbX
  \to
\ProbY
$
be a
$\Prob$-arrow
and
$
v
  \in
\mathcal{L}^{1}(\ProbNY)
$
which is independent of
$
f^-
$.
Then we have,
\begin{equation}
\label{eq:f_inde_th}
E^{f^-}(v)
   \sim_{\mathbb{P}_X}
\mathbb{E}^{\mathbb{P}_Y}[v]
E^{f^-}(1_Y) .
\end{equation}

\end{thm}
\begin{proof}
Let
$
\{
  u_n
    :
  \mathbb{R}
    \to
  \mathbb{R}
\}_{n \in \mathbb{N}}
$
be a sequence of functions defined by
$
u_n
  :=
id_{\mathbb{R}} \cdot 1_{[-n,n]}
$.
Then
by Theorem \ref{thm:indep},
$
E^{f^-}(u_n \circ v)
  \sim_{\mathbb{P}_X}
\mathbb{E}^{\mathbb{P}_Y}[
  u_n \circ v
]
E^{f^-} (1_Y)
$
since
$
u_n
  \in
L^{\infty}(
  \bar{\mathbb{R}}_{u_n}
)
$.
On the other hand,
\if \obsolete 1
{\color{red}
$
u_n
   \to
id_{\mathbb{R}}
$
as
$n$
goes to
$\infty$.
} 
\fi 
we have
\[
 u_n \circ v = u_n \circ v_+ - u_n \circ v_-
\]
and
\begin{align*}
 0 &\le u_n \circ v_+ \uparrow v_+ , \\
 0 &\le u_n \circ v_- \uparrow v_-
\end{align*}
as $n$ goes to $\infty$, where $v_+$ ($v_-$) is the positive (negative) part of the $\mathbb{R}$-valued function $v$.
So by
Proposition \ref{prop:lin}
and
Proposition \ref{prop:mc},
we obtain
\begin{align*}
 E^{f^-} (v) &\sim_{\mathbb{P}_X} \lim_{n \to \infty} E^{f^-} (u_n \circ v) \\
 &
 \sim_{\mathbb{P}_X} \lim_{n \to \infty} \left( E^{f^-} (u_n \circ v_+) - E^{f^-} (u_n \circ v_-) \right)
 \\
 &
 \sim_{\mathbb{P}_X} \lim_{n \to \infty} \mathbb{E}^{\mathbb{P}_Y}[u_n \circ v_+] E^{f^-} (1_Y) - \lim_{n \to \infty} \mathbb{E}^{\mathbb{P}_Y}[u_n \circ v_-] E^{f^-} (1_Y)
 \\
 &
 \sim_{\mathbb{P}_X} \mathbb{E}^{\mathbb{P}_Y}[v_+] E^{f^-} (1_Y) - \mathbb{E}^{\mathbb{P}_Y}[v_-] E^{f^-} (1_Y)
 \\
 &\sim_{\mathbb{P}_X} \mathbb{E}^{\mathbb{P}_Y}[v]
E^{f^-} (1_Y).
\qedhere
\end{align*}
\end{proof}

As a combination of
(\ref{eq:f_inde_th})
and
(\ref{eq:uncond_val}),
we have
\begin{equation}
\label{eq:rel_cond_uncond_when_ind}
E^{f^-}(v)
   \sim_{\mathbb{P}_X}
E^{!_Y^-}(v)(*)
E^{f^-}(1_Y) ,
\end{equation}
which is a natural generalization of
the relationship between 
classical conditional expectations
given independent 
$\sigma$-fields and
unconditional expectations.

\section{Completion Functor}
\label{sec:complF}


The following definition is taken from 
pages 202-203 of
\cite{williams1991}.

\begin{defn}{\cite{williams1991}}
\label{defn:comp_sp}
Let
$
(X, \Sigma_X, \mathbb{P}_X)
$
be a probability space.
\begin{enumerate}
\item
$
\Sigma_X^*
  := \{
F \subset X
  \mid
\exists
  A, B \in \Sigma_X,
A \subset F \subset B
  \; \mathrm{and} \;
\mathbb{P}_X( B - A) = 0
\}
$,

\item
For
$
F \in \Sigma_X^*
$,
$
\mathbb{P}_X^*(F)
$
is defined by
$
\mathbb{P}_X^*(F)
  :=
\mathbb{P}_X(A)
  =
\mathbb{P}_X(B)
$,
where
$
A, B
  \in
\Sigma_X
$
satisfies
$
A \subset F \subset B
$
and
$
\mathbb{P}_X( B - A) = 0
$.

\end{enumerate}

\end{defn}

Then, it is well-known that the triple
$
(X, \Sigma_X^*, \mathbb{P}_X^*)
$
is well-defined and
becomes a probability space
called the
\newword{completion}
of
$
(X, \Sigma_X, \mathbb{P}_X)
$.

\begin{prop}
\label{prop:comp_f_lemma}
Let 
$
f^-
  :
(X, \Sigma_X, \mathbb{P}_X)
  \to
(Y, \Sigma_Y, \mathbb{P}_Y)
$
be a
$\Prob$-arrow.
\begin{enumerate}
\item
The function
$
f : Y \to X
$
is
$
\Sigma_Y^*
  /
\Sigma_X^*
$-measurable.

\item
$
\mathbb{P}_Y^* \circ f^{-1}
  \ll
\mathbb{P}_X^*
$.

\end{enumerate}
\end{prop}
\begin{proof}
\begin{enumerate}
\item
For any
$
F \in \Sigma_X^*
$,
by
Definition \ref{defn:comp_sp}
there exist
$
A, B
  \in
\Sigma_X
$
such that
$
A \subset F \subset B
$
and
$
\mathbb{P}_X(B - A) = 0
$.
Then, since
$
\mathbb{P}_Y \circ f^{-1}
  \ll
\mathbb{P}_X
$,
we have
$
f^{-1}(A)
 \subset
f^{-1}(F)
 \subset
f^{-1}(B)
$
and
$
\mathbb{P}_Y(
  f^{-1}(B)
-
  f^{-1}(A)
)
  =
\mathbb{P}_Y(
  f^{-1}(
    B - A
  )
)
  =
0
$.
Therefore, again by
Definition \ref{defn:comp_sp},
$
  f^{-1}(F)
\in
  \Sigma_Y^*
$.

\item
Assume that
$
F \in \Sigma_X^*
$
and
$
\mathbb{P}_X(F) = 0
$.
Then,
it is sufficient to show that
$
(\mathbb{P}_Y^* \circ f^{-1})(F) = 0
$.
Now by
Definition \ref{defn:comp_sp},
there exists
$
B \in \Sigma_X
$
such that
$
F \subset B
$
and
$
\mathbb{P}_X^*(F)
  =
\mathbb{P}_X(B)
  =
0
$.
Then by
$
\mathbb{P}_Y \circ f^{-1}
  \ll
\mathbb{P}_X
$,
we have
$
(
\mathbb{P}_Y^* 
  \circ
f^{-1}
)(F)
  \le
(
\mathbb{P}_Y^*
  \circ
f^{-1}
)(B)
  =
(
\mathbb{P}_Y 
  \circ
f^{-1}
)(B)
  =
0
$.
\qedhere
\end{enumerate}
\end{proof}

Proposition \ref{prop:comp_f_lemma}
makes the following definition
well-defined.

\begin{defn}{[Functor $\mathcal{C}$]}
\label{defn:fun_C}
A functor
$
\mathcal{C} : \Prob \to \Prob
$
is defined by:
\begin{equation*}
\xymatrix@C=15 pt@R=30 pt{
    X
  &
    (X, \Sigma_X, \mathbb{P}_X)
     \ar @{->}_{f^-} [d]
     \ar @{|->}^{\mathcal{C}} [rr]
  &&
    \mathcal{C}(X, \Sigma_X, \mathbb{P}_X)
     \ar @{}^-{:=} @<-6pt> [r]
     \ar @{->}^{
       \mathcal{C} f^-
     } [d]
  &
    (X, \Sigma_X^*, \mathbb{P}_X^*)
     \ar @{->}_{f^-} [d]
\\
    Y
      \ar @{->}^f [u]
  &
    (Y, \Sigma_Y, \mathbb{P}_Y)
     \ar @{|->}^{\mathcal{C}} [rr]
  &&
    \mathcal{C}(Y, \Sigma_Y, \mathbb{P}_Y)
     \ar @{}^-{:=} @<-6pt> [r]
  &
    (Y, \Sigma_Y^*, \mathbb{P}_Y^*)
}
\end{equation*}
\end{defn}

The functor
$\mathcal{C}$
is called a
\newword{completion functor}.

\end{document}